\documentclass[12pt,nopagetitle]{article}
\usepackage{amsfonts, amsmath, amssymb, amsthm}  % различная математика
\usepackage[english]{babel}
\usepackage[utf8]{inputenc}
\usepackage[OT1]{fontenc}
\usepackage{bm}

\usepackage{textcomp, cmap, comment}	
\usepackage{graphicx, wrapfig}
\usepackage{epigraph, xcolor, hyperref}
\usepackage{array,tabularx,tabulary,booktabs}
\usepackage{euscript, mathrsfs}	

\newcounter{iii}

\newcommand{\bb}{{\mathcal B}}
\newcommand{\aaa}{{\mathcal A}}

\newcommand{\ff}{\mathcal F}
\theoremstyle{plain}
\newtheorem{thm}{Theorem}

\newtheorem{lem}{Lemma}
\newtheorem{pro}{Problem}
\newtheorem{cor}{Corollary}
\theoremstyle{definition}
\newtheorem{defi}{Definition}

\newtheorem{thrm}{Theorem}

\title{Regular bipartite graphs and intersecting families}
\author{Andrey Kupavskii\footnote{Moscow Institute of Physics and Technology, Ecole Polytechnique F\'ed\'erale de Lausanne; Email: {\tt kupavskii@yandex.ru} \ \ Research of Andrey Kupavskii is supported by the grant RNF~16-11-10014.}, Dmitriy Zakharov}
\date{}

\begin{document}
\maketitle
\begin{abstract} In this paper we present a simple unifying approach to prove several statements about intersecting and cross-intersecting families, including the Erd\H os--Ko--Rado theorem, the Hilton--Milner theorem, a theorem due to Frankl concerning the size of intersecting families with bounded maximal degree, and versions of results on the sum of sizes of non-empty cross-intersecting families due to Frankl and Tokushige. Several new stronger results are also obtained.

Our approach is based on the use of regular bipartite graphs. These graphs are quite often used in Extremal Set Theory problems, however, the approach we develop proves to be particularly fruitful.

\end{abstract}
\section{Intersecting families}\label{sec1}
Let $[n]:=\{1,\ldots,n\}$ denote the standard $n$-element set. The family of all $k$-element subsets of  $[n]$ we denote by ${[n]\choose k}$, and the set of all subsets of $[n]$ we denote by  $2^{[n]}$. Any subset of $2^{[n]}$ we call a \textit{family}. We say that a family is \textit{intersecting}, if any two sets from the family intersect.

Probably, the first theorem, devoted to intersecting families, was the famous  theorem of Erd\H os, Ko and Rado:
\begin{thrm}[Erd\H os, Ko, Rado, \cite{EKR}]\label{ekr} Let $n\ge 2k>0.$ Then for any intersecting family  $\ff\subset {[n]\choose k}$ one has $|\ff|\le {n-1\choose k-1}$.
\end{thrm}

It is easy to give an example of an intersecting family, on which the bound from Theorem \ref{ekr} is attained: take the family of all $k$-element sets containing element $1$.

%Moreover, it is known that up to renumbering of the elements of the ground set, for $n>2k$ this is the only intersecting family of $k$-sets of cardinality  ${n-1\choose k-1}$.

Any family, in which all sets contain a fixed element, we call  \textit{trivially intersecting}. What size can an intersecting family have, provided that it is not trivially (nontrivially) intersecting? For  $n=2k$ it is easy to construct many intersecting families of size ${2k-1\choose k-1}$ by choosing exactly one $k$-set out of each two complementary sets. For $n>2k$ the answer is given by the Hilton-Milner theorem.

\begin{thrm}[Hilton, Milner, \cite{HM}]\label{hm} Let $n> 2k$ and $\mathcal{F}\subset {[n]\choose k}$ be a nontrivially intersecting family. Then $|\mathcal{F}|\le {n-1\choose k-1}-{n-k-1\choose k-1}+1$.
\end{thrm}

Let us mention that the upper bound is attained on the family that contains exactly one  $k$-element set $A$, such that $1\notin A$, and all $k$-element sets containing $1$ and intersecting $A$.

Hilton--Milner theorem shows in a quite strong sense that the example for the Erd\H os--Ko--Rado theorem is unique, moreover, that Erd\H os--Ko--Rado theorem is stable in the sense that if any intersecting family is large enough, then it must be a trivially intersecting family. A theorem due to Frankl, discussed in Section \ref{sec3}, provides us with a much stronger stability result.

In this section we illustrate our new technique by giving a unified proof of Theorems \ref{ekr} and \ref{hm}. The results on cross-intersecting families are discussed in Section~\ref{sec2}, and Frankl's theorem is discussed in Section~\ref{sec3}.\\

The first part of the proof is fairly standard. We show that we may simplify the structure of the families using \textit{shifting}.\\

For a given pair of indices $1\le i<j\le n$ and a set $A \in 2^{[n]}$ define its {\it $(i,j)$-shift} $S_{i,j}(A)$ as follows. If $i\in A$ or $j\notin A$, then $S_{i,j}(A) = A$. If $j\in A, i\notin A$, then $S_{i,j}(A) := (A-\{j\})\cup \{i\}$. That is, $S_{i,j}(A)$ is obtained from $A$  by replacing $j$ with $i$.

Next, define the  $(i,j)$-shift $S_{i,j}(\mathcal F)$ of a family $\mathcal F\subset 2^{[n]}$:

$$S_{i,j}(\mathcal F) := \{S_{i,j}(A): A\in \mathcal F\}\cup \{A: A,S_{i,j}(A)\in \mathcal F\}.$$

We call a family $\mathcal F$ \textit{shifted}, if $S_{i,j}(\mathcal F) = \mathcal F$ for all $1\le i<j\le n$. Note that shifting does not change the cardinality of the family. Note also that we always shift to the left, replacing the larger element with the smaller one.\\

%Затем, с помощью паросочетаний во вспомогательных регулярных двудольных графах, мы последовательно переведем  изначальное семейство в подмножество экстремального семейства из Теоремы \ref{ekr} или \ref{hm}. Перейдем к доказательству теорем.\\

Consider a family $\ff$ of maximal size, satisfying the conditions of Theorem \ref{hm} or  \ref{ekr}. We prove that we may suppose that $\mathcal{F}$ is shifted. It is easy to check that an intersecting family stays intersecting after shifts. This is already sufficient for the Erd\H os--Ko--Rado theorem. In the case of the Hilton--Milner theorem we must make sure that the shifted family we obtain is nontrivially intersecting. For that it is sufficient to do several shifts of  $\mathcal{F}$ and renumberings of the ground set, that result in  ${[k+1]\choose k}\subset \mathcal{F}$. The subfamily ${[k+1]\choose k}$ is invariant under shifts, thus its presence guarantees that $\ff$ will stay nontrivially intersecting after any subsequent shifts.

Do the shifts of $\mathcal{F}$ until either the family is shifted or any new shift will result in $\ff$ becoming trivially intersecting. If the former happens, we end up having a nontrivially intersecting shifted family. And if the latter happens, then after the last shift all the sets of $\mathcal F$ intersect a certain two-element subset $\{x, y\}$. Renumber the elements of $[n]$, so that this pair becomes a pair $\{1, 2\}$. Since $\ff$ has maximal size, it contains all the $k$-element sets, containing both $1$ and $2$. Due to this,  the $(i,j)$-shifts, where $i,j>2$, do not affect the property of $\ff$ to be notrivially intersecting. Therefore, we may assume that $\ff$ is invariant under these shifts.
 Since not all sets in $\ff$ contained $1$ (and not all sets contained $2$), then $\ff$ must contain the sets $\{2,\ldots, k+1\}$ and $\{1, 3, \ldots, k+1\}$.  The condition ${[k+1]\choose k}\subset \mathcal{F}$ is now fulfilled, and thus no shift can affect the nontriviality of the family. 

We note that, after preparing an initial version of the manuscript, we were told by Peter Frankl that exactly the same reduction to a shifted family was used by him to prove the Hilton--Milner theorem in \cite{Fra3}. The remaining part of his proof is completely different from our proof. A somewhat similar use of shifting technique to prove the Hilton--Milner theorem appeared in \cite{FF1}.
\\

Next we pass to the part of the proof where the bipartite graphs are employed.
\begin{defi} For given $A,B\in {[n]\choose k}$ we say that $A$ {\it lexicographically precedes}  $B$, or $A\prec B$, if the smallest element of $A\setminus B$ is smaller than that of $B\setminus A$.
\end{defi}

The \textit{order} of a $k$-element set is its number in the lexicographic order on ${[n]\choose k}$. The \textit{order} of a family $\ff\subset {[n]\choose k}$ is the maximal order among its elements. Among the families of maximal size satisfying the conditions of Theorem \ref{ekr} or \ref{hm} choose a shifted family $\mathcal{F}$ of minimal order. Consider the set $F\in \mathcal{F}$ of maximal order. By the definition of order, if $\ff$ is non-trivial, then $F$ does not contain $1$.

Note that for some $l\ge 1$ one has $|F\cap [2l-1]|=l$, otherwise one of the images of $F$ under several shifts is disjoint from $F$ (this property was first noted by Frankl in \cite{Fra2}). Take maximal such $l$, and put $ L=[2l]\setminus F$. Note that $|L|=l$. We remark that, unless $\mathcal F$ is trivial, we have $l\ge 2$. Consider the families $\mathcal{A}:=\{A\in\binom{[n]}{k}:A\cap [2l]=F\cap [2l]\}$ and $\mathcal{B}:=\{B\in\binom{[n]}{k}:B\cap [2l]=L\}$. The following lemma is the key to the proof.

\begin{lem}\label{lembp} Let $n>2k>0$ and $\ff$ be a nontrivially intersecting shifted family. Suppose that $F\in \ff$ has maximal order. Then, in the notations above, the family $\ff':=(\ff\setminus \aaa)\cup \bb$ is intersecting and has smaller order than $\ff$. Moreover, $|\ff'|\ge |\ff|$.
\end{lem}
\begin{proof}
Let us first show that for any $A,B$, where $B\in \mathcal{B},\ A\in \mathcal{F}\setminus \mathcal{A},$ we have $ A\cap B\not = \emptyset$. The contrary may happen only for $A\in \mathcal{F}$ satisfying $A\cap L = \emptyset$. Since $F$ has the maximal order in  $\ff$, we have $F\cap [2l]\subset A$, and, consequently, $A\in \mathcal{A}$. Therefore, $\mathcal{F}'$ is an intersecting family.

Since all sets from $\bb$ contain $1$, it is clear that the order of  $\ff'$ is smaller than that of $\ff$. Finally, we have $|\mathcal{A}|=\binom{n-2l}{k-l}=|\mathcal{B}|$. Consider a bipartite graph with parts  $\mathcal A,\mathcal B$, and edges connecting disjoint sets. Independent sets in this graph correspond to intersecting subfamilies of  $\mathcal A\cup\mathcal B$. This graph is regular, and, therefore, it contains a perfect matching. Consequently, the largest independent set has size $|\mathcal{B}|$. Therefore $|\mathcal{F}'|\ge|\mathcal{F}|$. \end{proof}

In the case of the  Erd\H os--Ko--Rado theorem, since the family $\ff$ had the smallest order among the families of maximal size, we conclude that $\ff$ must be trivial.\\

In the case of the Hilton--Milner theorem we obtain a contradiction between the properties of $\ff$ and Lemma \ref{lembp}, unless $F = \{2,\ldots, k+1\}$. Indeed, if $F \ne \{2,\ldots, k+1\}$, then $\{2,\ldots,k+1\}\in \ff'$, that is, $\ff'$ is nontrivially intersecting. Therefore, $F=\{2,\ldots, k+1\}$, we have $l=k$, and all sets in $\ff$, different from $F$, contain $1$. Both theorems are proved.\\

We remark that bipartite graphs have been used many times in Extremal Set Theory, in particular, in the proof the famous Sperner theorem. The Sperner theorem states that the size of the largest family on $[n]$ with no two sets containing each other is ${n\choose \lceil n/2\rceil}$. Sperner used matchings in biregular bipartite graphs between ${[n]\choose k}$ and ${[n]\choose k+1}$, where the edges connected sets, one of which contained the other. With the proof above, both cornerstones of Extremal Set Theory have proofs based on matchings in regular bipartite graphs.
\section{Cross-intersecting families}\label{sec2}
We call two families $\mathcal A,\mathcal B$ \textit{cross-intersecting}, if for any $A\in \mathcal A,B\in \mathcal B$ we have $A\cap B\ne \varnothing$. Cross-intersecting families are very useful in the problems on intersecting families, as we will illustrate in Section \ref{sec3}. In this section we discuss and prove some important properties of pairs of cross-intersecting families.

%/////////Обсуждение следствий из той же техники. //////

\begin{defi} Let $0<k<n$ and  $0\le m\le {n\choose k}$. Denote $\mathcal L(k,m)$ the family of the first $m$ sets from ${[n]\choose k}$ in the lexicographic order.
\end{defi}

The following theorem, proved independently by Kruskal \cite{Kr} and  Katona \cite{Ka}, is central in Extremal Set Theory.
\begin{thrm}[Kruskal \cite{Kr}, Katona \cite{Ka}]\label{thkk} If the families $\mathcal A\subset{[n]\choose a},\mathcal B\subset {[n]\choose b}$ are cross-intersecting, then  $\mathcal L(a,|\mathcal A|),\mathcal L(b,|\mathcal B|)$ are cross-intersecting as well.
\end{thrm}
We remark that it implies immediately that for any such cross-intersecting families either $|\aaa|\le {n-1\choose a-1}$, or $|\bb|\le {n-1\choose b-1}$.

The Kruskal--Katona theorem is ubiquitous in studying cross-intersecting families. One of the types of problems concerning cross-intersecting families is to bound the sum of the cardinalities of the two families. In this paper we give a proof for the following general theorem of this type.

\begin{thm}\label{lem2} Let $n \ge a+b$, and suppose that families $\mathcal F\subset{[n]\choose a},\mathcal G\subset{[n]\choose b}$ are cross-intersecting.  Fix an integer $j\ge 1$. Then the following inequality holds in three different assumptions, listed below:
\begin{equation}\label{eqlem1}|\mathcal F|+|\mathcal G|\le {n\choose b}+{n-j\choose a-j}-{n-j\choose b}.\end{equation}
1. If $a< b$ and we have $|\mathcal F|\ge {n-j\choose a-j}$;\\
2. If $a\ge b$ and  we have ${n-j\choose a-j}\le |\mathcal F|\le {n+b-1-a\choose b-1}+{n+b-2-a\choose b-1}$;\\ 3. If $a\ge b$ and  we have ${n-a+b-3\choose b-3}+{n-a+b-4\choose b-3}\le |\mathcal F|\le {n-j\choose a-j}$.
\end{thm}

%, then
%\begin{equation}\label{eqlem2}|\mathcal F|+|\mathcal G|\le {n\choose b}+{n-i\choose a-i}-{n-i\choose b}.\end{equation}
It is not difficult to come up with an example of a cross-intersecting pair on which the bound \eqref{eqlem1} is attained: take $\mathcal F :=\{F\in{[n]\choose a}: [j]\subset F\}, \mathcal G:=\{ G\in {[n]\choose b}: G\cap [j]\ne \emptyset\}$. We also give the intuition for the bounds on $|\ff|$ in Points 2 and 3. If $\ff=\mathcal L(a,|\ff|)$, then the family $\ff$ of size attaining the upper bound in Point 2 consists of all  sets containing $[a-b+1]$ and all the sets that contain $[a-b]$ and $\{a-b+2\}$. Similarly, such family $\ff$ attaining the lower bound in Point 3 consists of all sets containing $[a-b+3]$ and all sets containing $[a-b+2]$ and $\{a-b+4\}$.

We give the proof of Theorem~\ref{lem2} in the end of this section.

We note that another characteristic that is well-studied for intersecting families is the product of the cardinalities. We refer the reader to the result due to Pyber \cite{P} and a recent refinement \cite{FK3}. 

Returning to the sum of cardinalities, the following theorem including Point 1 and some cases of Point 2 of Theorem \ref{lem2} was proven in \cite{FT}.

\begin{thrm}[Frankl, Tokushige, \cite{FT}]\label{lemft} Let $n > a+b$, $a\le b$, and suppose that families $\mathcal F\subset{[n]\choose a},\mathcal G\subset{[n]\choose b}$ are cross-intersecting. Suppose that for some real number $\alpha \ge 1$ we have  ${n-\alpha \choose n-a}\le |\mathcal F|\le {n-1\choose n-a}$. Then \begin{equation}\label{eqft}|\mathcal F|+|\mathcal G|\le {n\choose b}+{n-\alpha \choose n-a}-{n-\alpha \choose b}.\end{equation}
\end{thrm}
%\textbf{(Можем ли мы доказать это утверждение с помощью нашего метода?)}\\

On the one side, in Theorem~\ref{lemft} the parameter $\alpha$ may take real values, while in Theorem~\ref{lem2} the parameter $j$ is bound to be integer. On the other side, Theorem~\ref{lem2} deals with the case $a>b$ and, more importantly, even for $a=b$ the restriction on the size of  $\ff$ is weaker. This plays a crucial role in the proof of Frankl's theorem in the next section.

Let us obtain a useful corollary of Theorem~\ref{lemft}, overlooked by Frankl and Tokushige, which allows to extend it to the case $a>b$.

\begin{cor}\label{corft} Let $n > a+b$, $a> b$, and suppose that $\mathcal F\subset{[n]\choose a},\mathcal G\subset{[n]\choose b}$ are cross-intersecting. Suppose that for some real  $\alpha \ge a-b+1$ we have ${n-\alpha \choose n-a}\le |\mathcal F|\le {n-a+b-1\choose n-a}$. Then (\ref{eqft}) holds.
\end{cor}
\begin{proof} Applying Theorem \ref{thkk}, we may suppose that $\mathcal F = \mathcal L(a,|\mathcal F|), \mathcal G=\mathcal L(b,|\mathcal G|)$.
Because of the restriction on the size of $\mathcal F$ we have $[a-b+1]\subset F$ for any  $F\in \mathcal F$. Consider two families $\mathcal F':=\{F'\in {[a-b+1,n]\choose b}: F'\cup [a-b]\in \mathcal F\}$, \- $\mathcal G':=\{G'\in {[a-b+1,n]\choose b}: G\in \mathcal G\}$. We have $|\mathcal F'| = |\mathcal F|$, and we may apply Theorem~\ref{lemft} for $\mathcal F', \mathcal G'$ with $n':= n-a+b, \alpha':=\alpha-a+b$. We get that $|\mathcal F|+|\mathcal G'|\le {n-a+b\choose b}+{n-\alpha \choose n-a}-{n-\alpha \choose b}$. We have $|\mathcal G\setminus\mathcal G'|\le {n\choose b}-{n-a+b\choose b}$, and summing this inequality with  the previous one we get the statement of the corollary.
\end{proof}
\vskip+0.1cm

\begin{proof}[Proof of Theorem~\ref{lem2}] The proof of Point 1 of the theorem is a simplified version of the proof of Point 2, and thus we first give the proof of Point 2, and then specify the parts which are different in the proof of Point 1.

 The proof of the theorem uses bipartite graph considerations similar to the ones embodied in Lemma \ref{lembp}. We give two slightly different versions of it in the proof of Point 2 and Point 3.\\

\textbf{Point 2.} Let $\mathcal F,\mathcal G$ be a cross-intersecting pair with {\it minimal} cardinality of $\ff$ among the pairs satisfying the conditions of Point 2 and having maximal possible sum of cardinalities.

Due to Theorem~\ref{thkk} we may suppose that $\mathcal F = \mathcal L(a,|\mathcal F|), \mathcal G=\mathcal L(b,|\mathcal G|)$. Let $F\in \mathcal F$ be the set with the largest order in $\ff$. The bound on the cardinality of $|\mathcal F|$ implies $F\supset \{1,\ldots a-b\}$ and $F\cap \{a-b+1,a-b+2\}\ne \emptyset$. Take the largest $l\ge 1$, for which $|F\cap [a-b+2l]|= a-b+l$ is satisfied. Put
$L:= [a-b+2l]\setminus F$. Consider the families $\mathcal{A}:=\{A\in\binom{[n]}{a}:A\cap [a-b+2l]= F\cap [a-b+2l]\}$ and $\mathcal{B}:=\{B\in\binom{[n]}{b}:B\cap [a-b+2l]=L\}$.

The considerations in this and the next paragraph follow closely the argument from Lemma \ref{lembp}. For any $B\in \mathcal{B},\ F'\in \mathcal{F}\setminus \mathcal{A}$ we have $ F'\cap B\not = \emptyset$. Indeed, take $F''\in \mathcal{F}, F''\cap L = \emptyset$. Since $F$ has the largest order in $\mathcal F$, we have $F''\cap [a-b+2l]\subset F$, and, consequently, $F\in \mathcal{A}$.
This implies that the pair $\mathcal{F}\setminus \mathcal{A}, \mathcal G\cup \mathcal{B}$ is cross-intersecting.

Next, $|\mathcal{A}|=\binom{n-a+b-2l}{b-l}=|\mathcal{B}|$. Consider a bipartite graph with parts $\mathcal A,\mathcal B$ and edges connecting disjoint sets. Independent sets in this graph correspond to the cross-intersecting pairs of subfamilies of $\mathcal A$ and $ \mathcal B$. This is a regular graph, and, consequently, the maximal independent set has the size  $|\mathcal{B}|$. Therefore $|\mathcal{F}\setminus \mathcal{A}|+|\mathcal G\cup \mathcal{B}|\ge |\mathcal{F}|+|\mathcal G|$.

Finally, if $|\mathcal F|>{n-j\choose a-j},$ then $F\nsupseteq [j]$, and thus all the sets containing $[j]$ belong to $\ff\setminus \aaa$. Therefore, ${n-j\choose a-j}\le |\mathcal{F}\setminus \mathcal{A}|<|\mathcal F|$,  and we obtain a contradiction with the minimality of  $\mathcal  F$.\\

\textbf{Point 1.} Let $\mathcal F,\mathcal G$ be a cross-intersecting pair with {\it minimal} cardinality of $\ff$ among the pairs satisfying the conditions of Point 2 and having maximal possible sum of cardinalities.

 The bound on the cardinality of $|\mathcal F|$ implies $F\supset [1,j]$. Take the largest $l\ge 1$, for which $|F\cap [2l]|= l$ is satisfied. Put
$L:= [2l]\setminus F$. Consider the families $\mathcal{A}:=\{A\in\binom{[n]}{a}:A\cap [2l]= F\cap [2l]\}$ and $\mathcal{B}:=\{B\in\binom{[n]}{b}:B\cap [2l]=L\}$. The rest of the proof is the same.
\\

\textbf{Point 3.} Let $\mathcal F,\mathcal G$ be a cross-intersecting pair with \textit{maximal} cardinality of $\ff$ among the pairs of maximal sum of cardinalities satisfying the conditions of Point 3. We again w.l.o.g. suppose that $\mathcal F = \mathcal L(a,|\mathcal F|), \mathcal G=\mathcal L(b,|\mathcal G|)$.

Let $F\in {[n]\choose a}\setminus \mathcal F$ have the smallest order. Due to the lower bound on  $|\mathcal F|$, one of the following conditions hold: either $F\cap [a-b+4]=[a-b+2]$ or for some integer $1\le i\le a-b+1$ one has $F\cap [i+1] = [i]$. Indeed, if the latter condition does not hold, then $F\supset [a-b+2]$. But all the sets containing $[a-b+2]$ and at least one of $\{a-b+3,a-b+4\}$ are in $\ff$, and so $F\cap [a-b+4]=[a-b+2]$. Denote by $t$ the expression $a-b+4$ in the former case, and $i+1$ in the latter case. Put $L:=[t]\setminus F$ and consider two families: $\mathcal{A}:=\{A\in\binom{[n]}{a}:A\cap [t]= F\cap [t]\}$ and $\mathcal{B}:=\{B\in\binom{[n]}{b}:B\cap [t]=L\}$. As in the previous point, the pair  $\mathcal{F}\cup\mathcal{A}, \mathcal G\setminus \mathcal{B}$ is cross-intersecting. Moreover, it is easy to see that $a-|F\cap [t]|\ge b-|L|$ and $n-t\ge a-|F\cap [t]|+b-|L|=a+b-t$. Consequently, $|\mathcal A| = {n-t\choose a-|F\cap [t]|}\ge {n-t\choose b-|L|} = |\mathcal B|$. Therefore, arguing as in the previous point, we conclude that  $|\mathcal F| = {n-j\choose a-j}$.
\end{proof}

\section{The diversity of intersecting families}\label{sec3}
For a family $\ff$ the \textit{diversity} $\gamma(\ff)$ is the quantity $|\ff|-\Delta(\ff)$, where $\Delta(\ff):=\max_{i\in[n]}\big|\{F:i\in F\in \ff\}\big|$. This notion was studied, in particular, in \cite{LP}. In this section we discuss the connection between the diversity and the size of an intersecting family. The following theorem was de-facto proved by Frankl \cite{Fra1} for \textit{integer} $u$ (with the difference that Frankl uses $\Delta(\ff)$ instead of $\gamma(\ff)$ to bound the cardinality of the family):

\begin{thm}\label{thm1} Let $n>2k>0$ and $\ff\subset {[n]\choose k}$ be an intersecting family. Then, if $\gamma(\ff)\ge {n-u-1\choose n-k-1}$ for some real $3\le u\le k$, then \begin{equation}\label{eq01}|\ff|\le {n-1\choose k-1}+{n-u-1\choose n-k-1}-{n-u-1\choose k-1}.\end{equation}
\end{thm}

The bound from Theorem \ref{thm1} is sharp for integer $u$, as witnessed by the example  $\mathcal A:=\{A\in {[n]\choose k}:[2,u]\subset A\}\cup\{A\in{[n]\choose k}: 1\in A, [2,u]\cap A\ne \emptyset\}.$ We note that the case $u = k$ of Theorem~\ref{thm1} is precisely the Hilton--Milner theorem.

Below we give a new proof of Theorem~\ref{thm1}, which is much simpler than the proof of the integral case in \cite{Fra1}.  The proof makes use of the Kruskal--Katona theorem and the methods developed in the previous sections.

For any disjoint $I,J\subset [n]$ define $$\ff(\bar I J):=\{F\setminus J: F\in \ff, J\subset F, F\cap I=\varnothing\} \subset {[n]\setminus (I\cup J)\choose k-|J|}.$$

\begin{proof}[Proof of Theorem~\ref{thm1}] Let $\mathcal F$ have maximal cardinality among the families satisfying the condition of the theorem.\\

  \textbf{Case 1:} $\pmb{\gamma(\ff)\le {n-4\choose k-3}}$. Let $1$ be the most popular element in $\ff$. Consider the families $\ff(1)$ and $\ff(\bar 1)$. It is clear that $|\ff| = |\ff(1)|+|\ff(\bar 1)|$.  These two families are cross-intersecting and, moreover, we have  ${n-u-1\choose n-k-1}\le |\ff(\bar 1)|\le {n-4\choose n-k-1}$. Using Corollary~\ref{corft} for $n':=n-1, a := k, b:=k-1,$ we get the statement of the theorem.\\

\textbf{Case 2:} $ \pmb{{n-4\choose k-3}<\gamma(\ff)\le {n-3\choose k-2}+{n-4\choose k-2}}$. This case is treated analogously, with the only difference that instead of Corollary~\ref{corft} we use Point 2 of Theorem~\ref{lem2} with $n':=n-1, a := k, b:=k-1, j=3$.\\

\textbf{Case 3:}  $\pmb{\gamma(\ff)> {n-3\choose k-2}+{n-4\choose k-2}}$. When the diversity is large, we cannot proceed as in Part 2 of Theorem~\ref{lem2}, since we cannot guarantee that, after passing to the lexicographically minimal elements, the set  $L$ of maximal order will satisfy the necessary condition $|L\cap [2l]|=l$ for some $l\ge 1$. (We cannot neither apply nor imitate the proof of Part 2 of Theorem~\ref{lem2} in this case.) We will proceed differently. First we show that we may assume that  $\ff$ is shifted.

\begin{lem}\label{lemdiv} For any intersecting family $\ff\subset{[n]\choose k}$ and  any $i,j$, satisfying $|\ff(i)|\ge |\ff(j)|$, we have $\gamma(\ff)-\gamma(S_{i,j}(\ff))\le {n-3\choose k-2}$.
\end{lem}
\begin{proof} Indeed, it is easy to see that the families $\ff(i\bar j),\ff(j\bar i)$ are cross-intersecting. The assumption of the lemma implies $|\ff(i\bar j)|\ge |\ff(j\bar i)|$, and thus by Theorem~\ref{thkk} we have $|\ff(j\bar i)|\le {n-3\choose k-2}$ (see the remark after the theorem). On the other hand, after the  $(i,j)$-shift the degree of $i$  cannot increase by more than  $|\ff(j\bar i)|$, since it is only in the  sets from $\ff(j\bar i)$ that the element $j$ may be replaced by $i$. Therefore,   $\gamma(\ff)-\gamma(S_{i,j}(\ff))\le |\ff(j\bar i)|\le {n-3\choose k-2}$.
\end{proof}

Rearrange the elements in the order of decreasing degree, and do consecutively the $(1,j)$-shifts, $j=2,3,\ldots$, then the $(2,j)$-shifts, $j=3,4,\ldots$ etc., until either the family becomes shifted, or the diversity of the family becomes at most  ${n-3\choose k-2}+{n-4\choose k-2}$. We denote the obtained family by $\ff$ again.  In the latter case by Lemma~\ref{lemdiv} we have $\gamma(\ff)\ge {n-4\choose k-2}=\frac{n-k-1}{k-2}{n-4\choose k-3}>{n-4\choose k-3}.$ Therefore, this case is reduced to Case 2.\\

Finally, what if $\ff$ is shifted? We may suppose that\\
 1. $\gamma(\ff)>{n-3\choose k-2}+{n-4\choose k-2}$, \\
 2. $\mathcal F$ is shifted,\\
 3. $\ff$ has the smallest diversity among the families $\ff'$ of maximal size with $\gamma(\ff')\ge{n-3\choose k-2}$.\\
 The last inequality may look strange, since it does not coincide with the inequality defining Case 3. However, since we know that the theorem holds for families with diversity between ${n-3\choose k-2}$ and ${n-3\choose k-2}+{n-4\choose k-2}$ with $u=3$, we may include all potential families of maximal size with such diversity to the class of families in question.
 
  The element $1$ is the most popular among the sets in  $\ff$. Find the set $F\in \ff$ of maximal order and, in terms of the proof from Section~\ref{sec1}, apply Lemma~\ref{lembp} to get a family $\ff'=(\ff\setminus \aaa)\cup \bb$ of smaller diversity and at least as large as $\ff$. Note that, again in terms of Lemma~\ref{lembp}, $l\ge 2$ and thus $|\aaa|={n-2l\choose k-l}\le {n-4\choose k-2}.$  Therefore, we have $\gamma(\ff')\ge \gamma(\ff)- |\mathcal A| \ge {n-3\choose k-2}$, which is a contradiction with the choice of $\ff$.\end{proof}

\section{Conclusion}
Results of the type presented in this paper has proven to be useful in other questions concerning intersecting and cross-intersecting families. In particular, new results concerning the structure of intersecting and cross-intersecting families were obtained in \cite{FK8}, and degree versions of results on intersecting families were obtained in \cite{Kup10}. 

There are several questions that remain and that are important for applications. The following problem was proposed by Frankl: is it true that $\gamma(\ff)\le {n-3\choose k-2}$ for any $n\ge 3k$? This is easy to verify for shifted families. More generally, how large the diversity can be for different values of $n, 2k\le n\le 3k$? The following families are the natural candidates of families with the largest diversity for different ranges of $n$: $\ff_i:=\{F\in {[n]\choose k}: |F\cap [2i+1]|\ge i+1\}$.

Another problem driven by the applications is the following strengthening of the Erd\H os--Ko--Rado theorem:

\begin{pro} Given an intersecting family $\ff$ of $k$-sets of $[n]$ for $n > 2k$, prove that $\Delta(\ff)+C\gamma(\ff) \le {n-1\choose k-1}$ for the largest possible $C> 1$.
\end{pro}

% Рассмотрим семейства $\ff(1),\ff(\bar 1)$. Эти семейства являются кросс-пересекающимися и, в силу сдвинутости $\ff$, лексикографически максимальное множество $L\in \ff(\bar 1)$ обязано удовлетворять $|L\cap [2j]|=j$ для некоторого $j\ge 2$. (Равно как и все остальные множества из $\ff(\bar 1)$.) Поэтому, вместо того, чтобы переходить к лексикографически минимальным элементам, как это сделано в Пункте 2 Леммы \ref{lem2}, произведем ту же замену подсемейств, что и в Пункте 2 Леммы \ref{lem2}, но для самих семейств $\ff(1),\ff(\bar 1)$. Получим пару кросс-пересекающихся семейств, суммарная мощность которых будет не меньше,  но порядок при этом будет строго меньше. Получаем противоречие с выбором $\ff$. Теорема доказана полностью.

\textsc{Acknowledgements. } We would like to thank Peter Frankl and the anonymous referee for bringing several references to our attention, and for useful comments on the manuscript that helped to improve the presentation.

\end{document}